\documentclass[UKenglish]{lipics-v2019}

\usepackage{amsmath,amssymb,amsfonts}
\usepackage{thm-restate}
\declaretheorem[name=Theorem,style=plain,sibling=theorem]{ourtheorem}
\declaretheorem[name=Corollary,sibling=theorem]{ourcorollary}
\declaretheorem[name=Lemma,Refname={Lemma,Lemmas},sibling=theorem]{ourlemma}
\declaretheorem[name=Proposition,Refname={Proposition,Propositions},sibling=theorem]{ourproposition}
\declaretheorem[name=Remark,Refname={Remark},style=plain,sibling=theorem]{ourremark}

\usepackage{tikz}
\usetikzlibrary{arrows,calc,automata}
\tikzset{LMC style/.style={>=stealth',every edge/.append style={thick},every state/.style={minimum size=10,inner sep=0}}}

\title
{On the Size of Finite Rational Matrix Semigroups}

\author{Georgina Bumpus}{University of Oxford, UK}{}{}{Work supported by St John's College, Oxford.}
\author{Christoph Haase}{University College London, UK}{}{https://orcid.org/0000-0002-5452-936X}{}
\author{Stefan Kiefer}{University of Oxford, UK}{}{}{Work supported by a Royal Society University Research Fellowship.}
\author{Paul-Ioan Stoienescu}{University of Oxford, UK}{}{}{Work supported by St John's College, Oxford.}
\author{Jonathan Tanner}{University of Oxford, UK}{}{}{Work supported by St John's College, Oxford.}

\authorrunning{G.~Bumpus, C.~Haase, S.~Kiefer, P.~Stoienescu, and J.~Tanner} 

\Copyright{Georgina Bumpus and Georgina Bumpus and Stefan Kiefer and Paul-Ioan Stoienescu and Jonathan Tanner}



\category{}

\supplement{}

\funding{
  This work is part of a project that has received funding from the
  European Research Council (ERC) under the European Union's Horizon
  2020 research and innovation programme (Grant agreement No.\ 852769,
  ARiAT).}

\acknowledgements{The authors thank James Worrell for
  discussing~\cite{Hrushovski18} with them.}

\nolinenumbers 

\hideLIPIcs  

\DeclareMathOperator{\rk}{rk}
\DeclareMathOperator{\im}{im}

\newcommand{\afmp}{afmp-$\Z$-VASS}
\newcommand{\A}{\mathcal{A}}
\newcommand{\B}{\mathcal{B}}
\newcommand{\C}{\mathbb{C}}
\newcommand{\F}{\mathbb{F}}
\newcommand{\GL}{\mathit{GL}}
\newcommand{\Gr}{\mathrm{Gr}}
\newcommand{\M}{\mathcal{M}}
\newcommand{\mon}[1]{\overline{#1}}
\newcommand{\N}{\mathbb{N}}
\newcommand{\ourbound}{2^{n (2 n + 3)} g(n)^{n+1}}
\newcommand{\Q}{\mathbb{Q}}
\newcommand{\R}{\mathbb{R}}
\newcommand{\reach}{\mathcal{R}}
\newcommand{\spa}[1]{\left< #1 \right>}
\newcommand{\triv}{\{\vec{0}\}}
\newcommand{\vzero}{\vec{0}}
\newcommand{\V}{\mathcal V}
\newcommand{\Z}{\mathbb{Z}}

\EventEditors{Artur Czumaj, Anuj Dawar, and Emanuela Merelli}
\EventNoEds{3}
\EventLongTitle{47th International Colloquium on Automata, Languages, and Programming (ICALP 2020)}
\EventShortTitle{ICALP 2020}
\EventAcronym{ICALP}
\EventYear{2020}
\EventDate{July 8--11, 2020}
\EventLocation{Saarbrücken, Germany (virtual conference)}
\EventLogo{}
\SeriesVolume{168}
\ArticleNo{115}

\begin{document}

\maketitle              

\begin{abstract}
  Let $n$ be a positive integer and $\mathcal M$ a set of rational
  $n \times n$-matrices such that $\mathcal M$ generates a
  finite multiplicative semigroup. We show that any matrix in the
  semigroup is a product of matrices in $\mathcal M$ whose length is at most
  $2^{n (2 n + 3)} g(n)^{n+1} \in 2^{O(n^2 \log n)}$, where $g(n)$ is the maximum order
  of finite groups over rational $n \times n$-matrices.
  This result implies algorithms with an elementary
  running time for deciding finiteness of weighted automata
  over the rationals and for deciding reachability in affine integer
  vector addition systems with states with the finite monoid property.
\end{abstract}

\section{Introduction} \label{sec-intro}

\paragraph*{The Burnside Problem}

An element~$g$ of a semigroup~$G$ is called \emph{torsion} if $g^i = g^j$ holds for some naturals $i < j$, and $G$ \emph{torsion} if all its elements are torsion. Burnside~\cite{Burnside02} asked in~1902 a question which became known as the \emph{Burnside problem} for groups: is every finitely generated torsion group finite?
Schur~\cite{Schur11} showed in~1911 that this holds true for groups of invertible complex matrices, i.e., any finitely generated torsion subgroup of~$\GL(n,\C)$ is finite.
This was generalised by Kaplansky~\cite[p.~105]{Kaplansky72} to matrices over arbitrary fields.
The Burnside problem for groups has a negative answer in general: in~1964 Golod and Shafarevich exhibited a finitely generated infinite torsion group~\cite{Golod64,GolSha64}.

\paragraph*{The Maximal Order of Finite Matrix Groups}

Schur's result~\cite{Schur11} assures that finitely generated torsion matrix groups are finite, but does not bound the group order.
Indeed, it is easy to see that any finite cyclic group is isomorphic to a group generated by a matrix in $\GL(2,\R)$.
The same is not true for $\GL(n,\Q)$:
An elementary proof 
(which we reproduce in the appendix following \cite{KuzmanovichPavlichenkov02})
shows that any finite subgroup of~$\GL(n,\Q)$ is conjugate to a finite subgroup of~$\GL(n,\Z)$.
Another elementary proof shows that the order of any finite subgroup of~$\GL(n,\Z)$ divides~$(2 n)!$; see, e.g., \cite[Chapter~IX]{Newman72}.
Thus, denoting the order of the largest finite subgroup of~$\GL(n,\Q)$ by~$g(n)$, we have $g(n) \le (2 n)!$.
It is shown in a paper by Friedland~\cite{Friedland97} that $g(n) = 2^n n!$ holds for all sufficiently large~$n$.
This bound is attained by the group of signed permutation matrices.
Friedland's proof rests on an article by Weisfeiler~\cite{Weisfeiler84} which in turn is based on the classification of finite simple groups.
Feit showed in an unpublished manuscript~\cite{FeitUnp} that $g(n) = 2^n n!$ holds if and only if $n \in \N \setminus \{2,4,6,7,8,9,10\}$.\footnote{A list of the maximal-order finite subgroups of~$\GL(n,\Q)$ for $n \in \{2,4,6,7,8,9,10\}$ can be found in~\cite[Table~1]{Berry04}.}
Feit's proof relies on an unpublished manuscript~\cite{WeisfeilerUnp}, also based on the classification of finite simple groups, which Weisfeiler left behind before his tragic disappearance.

\paragraph*{Deciding Finiteness of Matrix Groups}

Bounds on group orders give a straightforward, albeit inefficient, way of deciding whether a given set of matrices generates a finite group: starting from the set of generators, enlarge it with products of matrices in the set, until either it is closed under product or the bound on the order has been exceeded.
One can do substantially better: it is shown in~\cite{BabaiBealsRockmore93} that, using computations on quadratic forms, one can decide in polynomial time if a given finite set of rational matrices generates a finite group.

\paragraph*{Deciding Finiteness of Matrix Semigroups}

The Burnside problem has a natural analogue for semigroups.
In~1975, McNaughton and Zalcstein~\cite{McNaughtonZalcstein75} positively solved the Burnside problem for matrix semigroups, i.e., they showed, for any field~$\F$, that any finitely generated torsion subsemigroup of~$\F^{n \times n}$ is finite, using the result for groups by Schur and Kaplansky as a building block.
From a computational point of view, McNaughton and Zalcstein's result suggests an approach for deciding finiteness of the semigroup generated by a given set of rational matrices: finiteness is recursively enumerable, by closing the set of generators under product, as described above for groups.
On the other hand, infiniteness is recursively enumerable by enumerating elements in the generated semigroup and checking each element whether it is torsion.
By the contrapositive of McNaughton and Zalcstein's result, if the generated matrix semigroup is infinite, it has a non-torsion element, witnessing infiniteness.
However, deciding whether a given matrix has finite order is nontrivial.
Only in~1980 did Kannan and Lipton \cite{KannanLipton80,KannanLipton86} show that the so-called \emph{orbit problem} is decidable (in polynomial time), implying an algorithm for checking whether a matrix has finite order.

Avoiding this problem, Mandel and Simon~\cite{MandelSimon77} showed in~1977 that there exists a function $f : \N^3 \to \N$ such that if $S$~is a finite subsemigroup of~$\F^{n \times n}$, generated by $m$ of its elements, and the subgroups of~$S$ have order at most~$g$, then $S$~has size (cardinality) at most~$f(n,m,g)$.
For rational matrices, one may use the function~$g(n)$ from above for~$g$.
By making, in a sense, McNaughton and Zalcstein's proof quantitative, Mandel and Simon explicitly construct such a function~$f$, which implies an algorithm, with bounded runtime, for deciding finiteness of a finitely generated rational matrix semigroup.
A similar result about the decidability of this problem was obtained independently and concurrently by Jacob~\cite{Jacob77}.

\paragraph*{Size Bounds}

Unlike the function~$g$ for rational matrix groups, Mandel and Simon's function~$f(n,m,g)$ depends on~$m$, the number of generators.
This is unavoidable: the semigroup generated by the set $\M_m := \left\{\begin{pmatrix} 0 & i\\0&0\end{pmatrix} : i \in \{0, \ldots, m-1\}\}\right\}$ is the set $\M_m$ itself, with $|\M_m| = m$ for any $m \in \N$.
Further, the growth in~$n$ of Mandel and Simon's~$f$ is, roughly, a tower of exponentials of height~$n$.
They write in~\cite[Section~3]{MandelSimon77}: ``However, it is likely that our upper bound [$f(n,m,g)$] can be significantly improved.''

In~\cite[Chapter~VI]{BerstelReutenauer}, Berstel and Reutenauer also show, for the rational case, the existence of a function in $n$ and~$m$ that bounds the semigroup size.
They write: ``As we shall see, the function [\ldots] grows extremely rapidly.''
An analysis of their proof shows that the growth of their function is comparable with the growth of Mandel and Simon's function.
A related approach is taken in~\cite{Straubing}.
Further proofs of McNaughton and Zalcstein's result can be found, e.g., in \cite{Lallement79,FreedmanGG97,deLucaVarricchio99,Steinberg12}, but they do not lead to better size bounds.

\paragraph*{Length Bounds}

In~1991, Weber and Seidl~\cite{WeberSeidlITA} considered semigroups over \emph{nonnegative} integer matrices.
Using combinatorial and automata-theoretic techniques, they showed that if a finite set $\M \subseteq \N^{n \times n}$ generates a finite monoid, then for any matrix~$M$ of that monoid there are $M_1, \ldots, M_\ell \in \M$ with $\ell \le \lceil e^2 n!\rceil - 2$ such that $M = M_1 \cdots M_\ell$; i.e., any matrix in the monoid is a product of matrices in~$\M$ whose \emph{length} is at most $\lceil e^2 n!\rceil - 2$.
Note that this bound does not depend on the number of generators.
Weber and Seidl also give an example that shows that such a length bound cannot be smaller than $2^{n-2}$.

Almeida and Steinberg~\cite{AlmeidaSteinberg09} proved in~2009 a length bound for rational matrices and expressing the zero matrix: if a finite set $\M \subseteq \Q^{n \times n}$ (with $n>1$) generates a \emph{finite} semigroup that includes the zero matrix~$0$, then there are $M_1, \ldots, M_\ell \in \M$ with $\ell \le (2 n - 1)^{n^2} - 1$ such that $0 = M_1 \cdots M_\ell$.
A length bound of $n^5$ for expressing the zero matrix was recently given in the nonnegative integer case~\cite{19KM-STACS}.
It is open whether there is a polynomial length bound for expressing the zero matrix in the rational case.


\paragraph*{Our Contribution}

We prove a $2^{O(n^2 \log n)}$ length bound for the rational case:
\begin{ourtheorem} \label{thm-main-intro}
Let $\M \subseteq \Q^{n \times n}$ be a finite set of rational matrices such that $\M$ generates a finite semigroup~$\mon{\M}$.
Then for any $M \in \mon{\M}$ there are $M_1, \ldots, M_\ell \in \M$ with $\ell \le \ourbound \in 2^{O(n^2 \log n)}$ such that $M = M_1 \cdots M_\ell$.
(Here $g(n) \le (2 n)!$ denotes the order of the largest finite subgroup of~$\GL(n,\Q)$.)
\end{ourtheorem}
The example by Weber and Seidl mentioned above shows that any such
length bound must be at least $2^{n-2}$. A length bound trivially implies a size bound, and \autoref{thm-main-intro} allows us to obtain the
first significant improvement over the fast-growing function of Mandel
and Simon.
\begin{ourcorollary}\label{cor-card-bound}
  Let $\M \subseteq \Q^{n\times n}$ be a finite set of $m$ rational
  matrices that generate a finite semigroup $\mon{\M}$. Then
  $|\mon{\M}| \le m^{2^{O(n^2 \log n)}}$.
\end{ourcorollary}

The proof of \autoref{thm-main-intro} is largely based on linear-algebra
arguments, specifically on the structure of a certain graph of vector
spaces obtained from $\M$. This graph was introduced and analysed
by Hrushovski et al.~\cite{Hrushovski18} for the computation of
the \emph{Zariski closure} of the generated matrix semigroup.

After the preliminaries (\autoref{sec-prelim}) and the proof of \autoref{thm-main-intro} (\autoref{sec-main}), we discuss applications in automata theory (\autoref{sec-applications}).
In particular we show that our result implies the first elementary-time algorithm for deciding finiteness of weighted automata over the rationals.

\section{Preliminaries} \label{sec-prelim}

We write $\N = \{0, 1, 2, \ldots\}$.
For a finite alphabet~$\Sigma$, we write $\Sigma^* = \{a_1 \cdots a_k : k \ge 0,\ a_i \in \Sigma\}$ and $\Sigma^+ = \{a_1 \cdots a_k : k \ge 1,\ a_i \in \Sigma\}$ for the free monoid and the free semigroup generated by~$\Sigma$. The elements of~$\Sigma^*$ are called \emph{words}.
For a word $w = a_1 \cdots a_k$, its \emph{length}~$|w|$ is~$k$.
We denote by~$\varepsilon$ the empty word, i.e., the word of length~$0$.
For $L \subseteq \Sigma^*$, we also write $L^* = \{w_1 \cdots w_k : k
\ge 0,\ w_i \in L\} \subseteq \Sigma^*$ and $L^+ = \{w_1 \cdots w_k : k
\ge 1,\ w_i \in L\} \subseteq \Sigma^*$.

We denote by $I_n$ the $n\times n$-identity matrix, and by~$\vzero$
the zero vector.
For vectors $v_1, \ldots, v_k$ from a vector space, we denote their span by $\spa{v_1, \ldots, v_k}$. In this article, we view elements
of $\Q^n$ as \emph{row} vectors.

For some $n \in \N \setminus \{0\}$, let $\M \subseteq \Q^{n \times
  n}$ be a finite set of rational matrices, generating a finite
semigroup~$\mon{\M}$.
For notational convenience, throughout the paper, we associate to~$\M$
an alphabet~$\Sigma$ with $|\M| = |\Sigma|$, and a  bijection $M :
\Sigma \to \M$ which we extend to the monoid morphism $M : \Sigma^*
\to \mon{\M}  \cup \{ I_n \}$.
Thus we may write $M(\Sigma)$ and $M(\Sigma^*)$ for $\M$ and~$\mon{\M}
 \cup \{ I_n \}$, respectively.

We often identify a matrix $A \in \Q^{n \times n}$ with its linear transformation $A : \Q^n \to \Q^n$ such that $x \mapsto x A$ for row vectors $x \in \Q^n$.
To avoid clutter, we extend linear-algebra notions from matrices to words, i.e., we may write $\im w$, $\ker w$, $\rk w$ for the image $\im(M(w)) = \Q^n M(w)$, the kernel $\ker(M(w)) = \{x \in \Q^n : x M(w) = \vzero\}$, and the rank of~$M(w)$.

If all matrices in $M(\Sigma)$ are invertible and $M(\Sigma^*)$ is finite, then $M(\Sigma^*)$ is a finite subgroup of~$\GL(n,\Q)$.
For $n \in \N$, let us write $g(n)$ for the size of the largest finite subgroup of~$\GL(n,\Q)$.
As discussed in the introduction, a non-trivial but elementary proof shows $g(n) \le (2 n)!$, and it is known that $g(n) = 2^n n!$ holds for sufficiently large~$n$.


\paragraph*{Exterior Algebra} \label{sub:exterior}

This brief introduction is borrowed and slightly extended from~\cite[Section~3]{Hrushovski18}.
Let $V$ be an $n$-dimensional vector space over a field~$\F$. 
(We will only consider $V = \Q^n$.)
For any $r \in \N$, let $\A_r$ denote the set of maps $B : V^r \to \F$ so that $B$ is linear in each argument and further $B(v_1, \ldots, v_r) = 0$ holds whenever $v_i = v_{i+1}$ holds for some $i \in \{1, \ldots, r-1\}$.
These conditions imply that swapping two adjacent arguments changes the sign, i.e.,
\[
  B(v_1, \ldots, v_{i-2}, v_{i-1}, v_{i+1}, v_i, v_{i+2}, v_{i+3}, \ldots, v_r) \ = \
 -B(v_1, \ldots, v_r)\,.
\]
These properties of~$\A_r$ imply that, given an arbitrary basis $\{e_1, \ldots, e_n\}$ of~$V$, any $B \in \A_r$ is uniquely determined by all $B(e_{i_1}, \ldots, e_{i_r})$ where $1 \le i_1 < i_2 < \ldots < i_r \le n$.
For any $v_1, \ldots, v_r \in V$, define the \emph{wedge product}
\[
 v_1 \wedge \cdots \wedge v_r : \A_r \to \F \quad \text{by} \quad (v_1 \wedge \cdots \wedge v_r)(B) = B(v_1, \ldots, v_r)\,.
\]
It follows from the properties of~$\A_r$ above that the wedge product
is linear in each argument: if $v_i = \lambda u + \lambda' u'$ then
\[
  \Big(\bigwedge_{1\le i \le k} v_i\Big)(B) = \lambda \Big(\bigwedge_{1\le j<i}
    v_j \wedge u \wedge \bigwedge_{i<j\le k} v_j\Big)(B) +
    \lambda' \Big(\bigwedge_{1\le j<i}
    v_j \wedge u' \wedge \bigwedge_{i<j\le k} v_j\Big)(B)
\]
Moreover, $(v_1 \wedge \cdots \wedge v_r)(B) = 0$ if $v_i = v_j$ holds for some $i,j$ with $i \ne j$.

For $r \in \N$ define $\Lambda^r V$ as the vector space generated by the length-$r$ wedge products $v_1 \wedge \cdots \wedge v_r$ with $v_1, \ldots, v_r \in V$.
For any basis $\{e_1, \ldots, e_n\}$ of~$V$, the set
$\{e_{i_1} \wedge \cdots \wedge e_{i_r} : 1 \le i_1 < \ldots < i_r \le n\}$
is a basis of $\Lambda^r V$; hence $\dim \Lambda^r V = \binom n r$.
Note that $\Lambda^1 V = V$ and $\binom n r = 0$ for $r > n$.
One can view the wedge product as an associative operation $\mathord{\wedge} : \Lambda^r V \times \Lambda^\ell V \to \Lambda^{r+\ell} V$.
Define the \emph{exterior algebra of~$V$} as the direct sum $\Lambda V = \Lambda^0 V \oplus \Lambda^1 V \oplus \cdots$.
Then also $\mathord{\wedge} : \Lambda V \times \Lambda V \to \Lambda V$.

It follows that for $u_1, \ldots, u_r \in V$, we have $u_1 \wedge \cdots \wedge u_r \ne \vzero$ if and only if $\{u_1, \ldots, u_r\}$ is linearly independent.
Furthermore, for $u_1, \ldots, u_r, v_1, \ldots, v_r \in V$ and $u = u_1 \wedge \cdots \wedge u_r \ne \vzero$ and $v = v_1 \wedge \cdots \wedge v_r \ne \vzero$, we have that $u,v$ are scalar multiples if and only if $\spa{u_1, \ldots, u_r} = \spa{v_1, \ldots, v_r}$.

The Grassmannian $\Gr(n)$ is the set of subspaces of~$\Q^n$.
By the above-stated properties of the wedge product there is an injective function
\[
\iota : \Gr(n) \to \Lambda\,\Q^n
\]
such that, for all~$W \in \Gr(n)$, we have $\iota(W) = v_1 \wedge \cdots \wedge v_r$ where $\{v_1, \ldots, v_r\}$ is an arbitrarily chosen basis of~$W$.
Note that the particular choice of a basis for~$W$ only changes the value of~$\iota(W)$ up to a constant.
Given subspaces $W_1, W_2 \in \Gr(n)$, we moreover have $W_1 \cap W_2 = \triv$ if and only if $\iota(W_1) \wedge \iota(W_2) \ne \vzero$.

\section{Proof of \autoref{thm-main-intro}} \label{sec-main}

It is convenient to state and prove our main result in terms of monoids rather than semigroups:%
\begin{ourtheorem}\label{thm-main}
Let $M : \Sigma^* \to \Q^{n \times n}$ be a monoid morphism whose image $M(\Sigma^*)$ is finite.
Then for any $w \in \Sigma^*$ there is $u \in \Sigma^*$ with $M(w) = M(u)$ and
\[
 |u| \ \le \ \ourbound \in 2^{O(n^2 \log n)}\,.
\]
\end{ourtheorem}
With this theorem at hand, \autoref{thm-main-intro} follows
immediately:
\begin{proof}[Proof of \autoref{thm-main-intro}]
  Let $M \in \mon\M$ be an element of the semigroup generated by
  $\M$. If $M \neq I_n$, by \autoref{thm-main}, $M$ can be written as a
  short product. Otherwise, $M = I_n \in G$, where
  $G=\mon{\M \cap \GL(n,\Q)}$ is a finite group of order at most
  $g(n)$. For any product $M_1\cdots M_\ell$ with $\ell > g(n)$, there
  are $1\le i<j\le \ell$ such that $M_1\cdots M_i = M_1\cdots M_j$,
  and so $M_1\cdots M_\ell = M_1\cdots M_iM_{j+1}\cdots
  M_\ell$. Hence, there are $\ell \in \{1, \ldots, g(n)\}$ and
  $M_1,\ldots,M_\ell \in \M$ such that $M = I_n = M_1\cdots M_\ell$.
\end{proof}

\begin{ourremark}\label{rem-group-diameter}
  The same argument as in the proof above shows that in a finite
  \emph{monoid} $(H, \mathord{\cdot})$, generated by $G \subseteq H$,
  for any $h \in H$ there are $\ell \in \{0,\ldots, |H| - 1\}$ and
  $g_1, \ldots, g_\ell \in G$ with
  $h = g_1 \cdots g_\ell$.
\end{ourremark}

In the remainder of this section, we prove \autoref{thm-main}.
We assume that
$M : \Sigma^* \to \Q^{n \times n}$ is a monoid morphism with finite
image~$M(\Sigma^*)$.

\subsection{The Maximum-Rank Case} \label{sub-maximum-rank-case}

In this subsection we prove:
\begin{ourproposition} \label{prop-max-rank}
Suppose that there is $r \le n$ with $\rk a = r$ for all $a \in \Sigma$.
Let $w \in \Sigma^*$ with $\rk w = r$.
Then there is $u \in \Sigma^*$ with $M(w) = M(u)$ and
\[
 |u| \ \le \ 2^{2 n + 3} g(n) - 1
 \ \in \ 2^{O(n \log n)}\,.
\]
\end{ourproposition}

In this subsection we assume that $\rk a = r$ holds for all $a \in \Sigma$.
For the proof of \autoref{prop-max-rank}, we define a directed labelled
graph
$G$ whose vertices are the vector spaces $\im w$ for $w \in \Sigma^*$
such that $\rk w = r$, and whose edges are triples $(V_1, a, V_2)$
such that $a \in \Sigma$ and $V_1 M(a) = V_2$.
Let $(V_1, a, V_2)$ be an edge; then $V_2 \subseteq \im a$, but $\dim V_2 = r = \rk a = \dim \im a$, hence $V_2 = \im a$, i.e., the edge label determines the edge target.
We will implicitly use the fact that any path in~$G$ is determined by its start vertex and the sequence of its edge labels.
Note that if $V_1$ is a vertex and $a \in \Sigma$, the edge $(V_1, a,
\im a)$ is present in $G$ if and only if $\rk V_1M(a) = r$ if and only if
$V_1 \cap \ker a = \triv$.

The following two lemmas, which are variants of lemmas in~\cite[Section~6]{Hrushovski18}, are statements about the structure of~$G$ in terms of its strongly connected components (SCCs).

\begin{ourlemma} \label{lem-number-of-SCCs}
Let $w = w_1 \cdots w_k$ for 
$w_1, \ldots, w_k \in \Sigma^+$ with $\rk w = r$ such that the $k$~vertices $\im w_1, \ldots, \im w_k$ are all in different SCCs of~$G$.
Then $k \le 2 \binom{n}{r}$.
\end{ourlemma}
\begin{proof}
Let $i \in \{2, \ldots, k-1\}$.
Since $\rk w_i = r = \rk (w_i w_{i+1})$, we have $\im w_i \cap \ker w_{i+1} = \triv$, thus $\iota(\im w_i) \wedge \iota(\ker w_{i+1}) \ne \vzero$.
On the other hand, for any $j < i$, since $\im w_i, \im w_j$ are in different SCCs and $\im w_i$ is reachable from~$\im w_j$, the vertex $\im w_j$ is not reachable from~$\im w_i$; therefore we have $\im w_i \cap \ker w_{j} \ne \triv$, thus $\iota(\im w_i) \wedge \iota(\ker w_{j}) = \vzero$.
It follows that $\iota(\ker w_{i+1}) \not\in \spa{\iota(\ker w_j) : j < i}$.
Indeed, if $\iota(\ker w_{i+1}) = \sum_{j<i} \lambda_j \iota(\ker w_j)$ for some~$\lambda_1, \ldots, \lambda_{i-1}$ then, by linearity of the wedge product, $\iota(\im w_i) \wedge \iota(\ker w_{i+1}) = \sum_{j<i} \lambda_j (\iota(\im w_i) \wedge \iota(\ker w_j)) = \vzero$, a contradiction.

We show by induction on $i$ that
$\dim \spa{\iota(\ker w_j) : j \in \{1, \ldots, i\}} \ge i/2$ for all
$i \in \{1, \ldots, k\}$. This is clear for $i=1,2$. For the induction
step, we have
$\dim \spa{\iota(\ker w_j) : j \in \{1, \ldots, i+1\}} \ge \dim \spa{
  \iota(\ker w_{i+1}), \iota(\ker w_j) : j \in \{1, \ldots, i-1\}} \ge
1 + (i-1)/2 = (i+1)/2$. Hence
$k/2 \le \dim \spa{\iota(\ker w_j) : j \in \{1, \ldots, k\}} \le \dim
\Lambda^{n-r} \Q^n = \binom{n}{r}$.
\end{proof}

\begin{ourlemma} \label{lem-diameter-of-SCC}
Let $a_1 \cdots a_k \in \Sigma^*$ be (the edge labels of) a shortest path in~$G$ from a vertex~$\im a_0$ to~$\im a_k$.
Then $k \le \binom{n}{r}$.
\end{ourlemma}
\begin{proof}
Let $i \in \{0, \ldots, k-2\}$.
We have $\im a_i \cap \ker a_{i+1} = \triv$, thus $\iota(\im a_i) \wedge \iota(\ker a_{i+1}) \ne \vzero$.
On the other hand, for any $j > i+1$, since $a_{i+1} \cdots a_j$ is a shortest path from $\im a_i$ to~$\im a_j$, there is no edge from $\im a_i$ to~$\im a_j$; therefore we have $\im a_i \cap \ker a_j \ne \triv$, thus $\iota(\im a_i) \wedge \iota(\ker a_j) = \vzero$.
It follows that $\iota(\ker a_{i+1}) \not\in \spa{\iota(\ker a_j) : j > i+1}$.

By induction it follows that $\dim \spa{\iota(\ker a_j) : j \in \{i+1, \ldots, k\}} \ge k-i$ holds for all $i \in \{0, \ldots, k-1\}$.
Hence $k \le \dim \spa{\iota(\ker a_j) : j \in \{1, \ldots, k\}} \le \dim \Lambda^{n-r} \Q^n = \binom{n}{r}$.
\end{proof}

The next lemmas discuss \emph{cycles} $w \in \Sigma^+$ in~$G$, i.e., (the edge labels of) paths in~$G$ such that $\im w \cap \ker w = \triv$.
A cycle~$w$ is said to be \emph{around} $\im w_0$ if $\im w = \im w_0$.
The following lemma says, loosely speaking, that cycles around a single vertex ``generate a group''.

\begin{ourlemma} \label{lem-cycle}
Let $w_0 \in \Sigma^+$ with $\rk w_0 = r$, and let $P \in \Q^{r \times n}$ be a matrix with $\im P = \im w_0$.
Then for every cycle $w \in \Sigma^+$ around~$\im w_0$ there exists a unique invertible matrix $M'(w) \in \GL(r,\Q)$ such that $P M(w) = M'(w) P$.
Moreover, for any nonempty set $C \subseteq \Sigma^+$ of cycles around~$\im w_0$, $M'(C^+)$ is a finite subgroup of~$\GL(r,\Q)$.
\end{ourlemma}
\begin{proof}
Let $w \in \Sigma^+$ be a cycle around~$\im w_0$.
Since $\im P \cap \ker (M(w)) = \triv$, it follows that $\im (P M(w)) = \im w = \im P$.
So the rows of $P M(w)$ are linear combinations of rows of~$P$, and vice versa, hence there is a unique $M'(w) \in \GL(r,\Q)$ with $P M(w) = M'(w) P$.

Let $C \subseteq \Sigma^+$ be a nonempty set of cycles around~$\im w_0$.
For any $w_1, w_2 \in C$ we have $M'(w_1 w_2) P = P M(w_1 w_2) = P M(w_1) M(w_2) = M'(w_1) P M(w_2) = M'(w_1) M'(w_2) P$, and since the rows of~$P$ are linearly independent, it follows that $M'(w_1 w_2) = M'(w_1) M'(w_2)$.
Thus, $M'(C^+)$ is a semigroup.

Towards a contradiction, suppose $M'(C^+)$ were infinite.
Since the rows of~$P$ are linearly independent, it follows that $M'(C^+) P$ is infinite, thus $P M(C^+)$ is infinite.
Since $\im w_0 = \im P$, there is a matrix $B \in \Q^{n \times r}$ with $M(w_0) = B P$.
Since the columns of~$B$ are linearly independent, the set $B P M(C^+)$ is infinite.
But this set equals $M(w_0 C^+)$, contradicting the finiteness of~$M(\Sigma^*)$.
Thus the semigroup $M'(C^+)$ is finite.
As $M'(C^+) \subseteq \GL(r,\Q)$, it follows that $M'(C^+)$ is a finite group.
\end{proof}

The following lemma allows us, loosely speaking, to limit the number of cycles in a word.

\begin{ourlemma} \label{lem-cycle-group}
Let $w_0, w_1, \ldots, w_k \in \Sigma^+$ such that $w_1, \ldots, w_k$ are cycles around~$\im w_0$.
Then there exist $\ell \le g(n) - 1$ and $\{u_1, \ldots, u_\ell\} \subseteq \{w_1, \ldots, w_k\}$ such that $M(w_0 w_1 \cdots w_k) = M(w_0 u_1 \cdots u_\ell)$.
\end{ourlemma}
\begin{proof}
We can assume $k \ge 1$.
Let $C = \{w_1, \ldots, w_k\}$.
Let $P$ and $M'(w)$ for $w \in C$ as in \autoref{lem-cycle}.
By \autoref{lem-cycle}, the set~$M'(C^+)$ is a finite subgroup of~$\GL(r,\Q)$, so we have $|M'(C^+)| \le g(r) \le g(n)$.
By \autoref{rem-group-diameter}, there are $\ell \le g(n) - 1$ and $u_1, \ldots, u_\ell \in C$ such that $M'(w_1) \cdots M'(w_k) = M'(u_1) \cdots M'(u_\ell)$.
Since $\im w_0 = \im P$, there is a matrix $B \in \Q^{n \times r}$ with $M(w_0) = B P$.
Hence we have $M(w_0 w_1 \cdots w_k) = B P M(w_1) \cdots M(w_k) = B M'(w_1) \cdots M'(w_k) P = B M'(u_1) \cdots M'(u_\ell) P = B P M(u_1) \cdots M(u_\ell) = M(w_0 u_1 \cdots u_\ell)$.
\end{proof}

The following lemma allows us to add cycles to a word.

\begin{ourlemma} \label{lem-cycle-identity}
Let $w \in \Sigma^+$ be a cycle in~$G$.
Then there exists $\rho(w) \in \N \setminus\{0\}$ such that $M(w_0) = M(w_0 w^{\rho(w)})$ holds for all $w_0 \in \Sigma^+$ with $\im w_0 = \im w$.
\end{ourlemma}
\begin{proof}
Let $P \in \Q^{r \times n}$ be a matrix with $\im P = \im w$.
By \autoref{lem-cycle}, there exists $M'(w) \in \GL(r,\Q)$ such that $P M(w) = M'(w) P$ and $\{M'(w)^i : i \in \N\}$ is a finite group.
Define~$\rho(w)$ to be the order of this group, i.e., $M'(w)^{\rho(w)} = I_r$.
Let $w_0 \in \Sigma^+$ with $\im w_0 = \im w$.
Since $\im w_0 = \im P$, there is a matrix $B \in \Q^{n \times r}$ with $M(w_0) = B P$.
Hence $M(w_0) = B P = B I_r P = B M'(w)^{\rho(w)} P = B P M(w)^{\rho(w)} = M(w_0) M(w)^{\rho(w)} = M(w_0 w^{\rho(w)})$.
\end{proof}

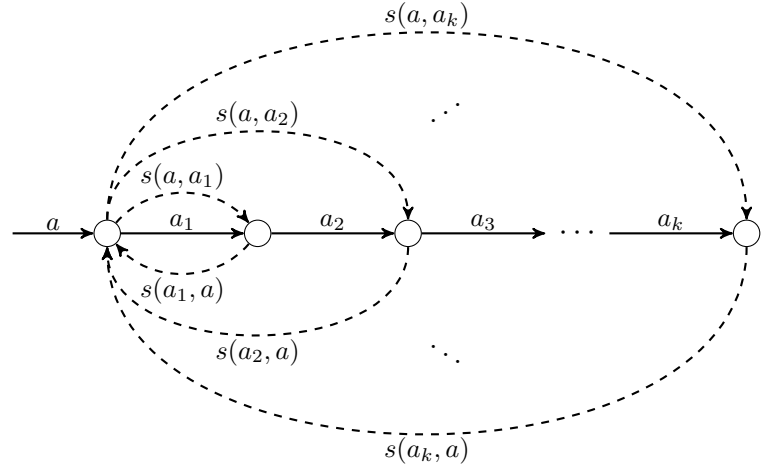
\begin{figure*}[t]
  \begin{center}
    \begin{tikzpicture}[LMC style,initial text = {},inner sep=1]

      \node (init) at (-1.3,-1) {};
      \node[state] (a) at (0,-1) {};
      \node[state] (a1) at (2,-1) {};
      \node[state] (a2) at (4,-1) {};
      \node[yscale=-1] (b3) at (4.5,+0.5) {\Large $\ \ddots\ $};
      \node (a3) at (6.25,-1) {\Large $\ \cdots\ $};
      \node (c3) at (4.5,-2.5) {\Large $\ \ddots\ $};
      \node[state] (ak) at (8.5,-1) {}; 
      
      \path[->] (init) edge node[above] {$a$} (a);

      \path[->] (a) edge node[above] {$a_1$} (a1);
      \path[->, bend left=50, dashed] (a) edge node[above] {$s(a,a_1)$} (a1);
      \path[->, bend left=50, dashed] (a1) edge node[below] {$s(a_1,a)$} (a);

      \path[->] (a1) edge node[above] {$a_2$} (a2);
      \path[->, bend left=90, dashed] (a) edge node[above] {$s(a,a_2)$} (a2);
      \path[->, bend left=90, dashed] (a2) edge node[below] {$s(a_2,a)$} (a);

      \path[->] (a2) edge node[above] {$a_3$} (a3);

      \path[->] (a3) edge node[above] {$a_k$} (ak);      
      \path[->, bend left=90, dashed] (a) edge node[above] {$s(a,a_k)$} (ak);
      \path[->, bend left=90, dashed] (ak) edge node[below] {$s(a_k,a)$} (a);
    \end{tikzpicture}
  \end{center}
  \caption{Illustration of the paths $w$ and $w'$ in
    \autoref{lem-modify-within-SCC}. Edges are depicted as solid
    arrows, paths as dashed arrows.}
\label{fig:modify-within-SCC}
\end{figure*}

The following lemma allows us to limit the length of paths within an SCC.

\begin{ourlemma} \label{lem-modify-within-SCC}
Let $a \in \Sigma$, and let $w \in \Sigma^*$ be a path in~$G$ from~$\im a$ such that $\im a$ and $\im w$ are in the same SCC.
Then there exists $u \in \Sigma^*$ with $M(a w) = M(a u)$ and
\[
 |u| \ \le \ 2^{n+2} g(n) - 2 \ \in \ 2^{O(n \log n)}\,.
\]
\end{ourlemma}
\begin{proof}
For any $b_1, b_2 \in \Sigma$ such that $\im b_1, \im b_2$ are in the SCC of~$\im a$, let $s(b_1, b_2) \in \Sigma^*$ be a shortest path from $\im b_1$ to~$\im b_2$.
By \autoref{lem-diameter-of-SCC}, we have $|s(b_1, b_2)| \le \binom n r$.

Suppose $w = a_1 \cdots a_k$ for $a_i \in \Sigma$.
For $i \in \{1, \ldots, k\}$ define the cycle $w_i := s(a_i,a) s(a,a_i)$ around $\im a_i$.
By \autoref{lem-cycle-identity}, we have $M(a w) = M(a w')$ for
\begin{align*}
 w' \ := \ & a_1 w_1^{\rho(w_1)} a_2 w_2^{\rho(w_2)} \cdots a_k w_k^{\rho(w_k)} \,.
\intertext{For $i \in \{1, \ldots, k\}$ also define the cycle $v_i := s(a,a_i) s(a_i,a)$ around $\im a$.
Then we have:}
 w' \ =  \ & a_1 s(a_1,a) v_1^{\rho(w_1)-1} s(a,a_1) a_2 s(a_2,a) v_2^{\rho(w_2)-1} s(a,a_2) \cdots a_k s(a_k,a) v_k^{\rho(w_k)-1} s(a,a_k)
\end{align*}
\autoref{fig:modify-within-SCC} illustrates the paths $w$ and $w'$.
Define a set of cycles $C \subseteq \Sigma^*$ around $\im a$ by
 \[ C \ := \ \{a_1 s(a_1,a), v_1, s(a,a_1) a_2 s(a_2,a), v_2, \ldots, s(a,a_{k-1}) a_k s(a_k,a), v_k\}\,.
 \]
Since $w' \in C^* s(a,a_k)$, by \autoref{lem-cycle-group}, there exist $\ell \le g(n)-1$ and $u_1, \ldots, u_\ell \in C$ such that $M(a w) = M(a w') = M(a u_1 u_2 \cdots u_\ell s(a,a_k))$.
For all $v \in C$ we have $|v| \le 2 \binom n r + 1 \le 2^{n+2}$, and $|s(a,a_k)| \le \binom n r \le 2^n$.
Hence the lemma holds for $u := u_1 u_2 \cdots u_\ell s(a,a_k)$, as $|u| \le 2^{n+2}(g(n)-1) + 2^n \le 2^{n+2} g(n) - 2$.
\end{proof}

We are ready to prove \autoref{prop-max-rank}.
\begin{proof}[Proof of \autoref{prop-max-rank}]
Decompose the word~$w$ into $w = a_1 w_1 a_2 w_2 \cdots a_k w_k$ for $a_i \in \Sigma$ so that for all $i \in \{1, \ldots, k\}$ the vertices $\im a_i, \im w_i$ are in the same SCC, and for all $i \in \{1, \ldots, k-1\}$ the vertices $\im w_i, \im a_{i+1}$ are in different SCCs.
By \autoref{lem-number-of-SCCs}, we have $k \le 2 \binom n r \le 2^{n+1}$.
For all $i \in \{1, \ldots, k\}$, by \autoref{lem-modify-within-SCC}, there is $u_i \in \Sigma^*$ with $|u_i| \le 2^{n+2} g(n) - 2$ such that $M(a_i w_i) = M(a_i u_i)$.
Hence the proposition holds for $u := a_1 u_1 a_2 u_2 \cdots a_k u_k$, as
$|u| \le 2^{n+1} (2^{n+2} g(n) - 2 + 1) \le 2^{2 n + 3} - 1$.
\end{proof}

\subsection{The General Case} \label{sub-general-case}

In this subsection we prove \autoref{thm-main}.
For $r \in \{0, \ldots, n\}$ let $d_r \in \N$ be the smallest number such that for any $w \in \Sigma^*$ with $\rk w \ge r$ there is $u \in \Sigma^*$ with $M(w) = M(u)$ and $|u| \le d_r$.
Also write $h$ for the bound from \autoref{prop-max-rank}.

\begin{ourproposition} \label{prop-induction-step}
For any $r \in \{0, \ldots, n-1\}$ we have $d_r \le d_{r+1} + (d_{r+1} + 1) h$.
\end{ourproposition}
\begin{proof}
Let $w \in \Sigma^*$ with $\rk w \ge r$.
We need to show that there is $u \in \Sigma^*$ with $M(w) = M(u)$ and $|u| \le d_{r+1} + (d_{r+1} + 1) h$.
Decompose~$w$ into $w = w_0 a_1 w_1 a_2 w_2 \cdots a_k w_k$ for $a_i \in \Sigma$ such that $\rk w_0 > r$ and for all $i \in \{1, \ldots, k\}$ we have $\rk (a_i w_i) = r$ and $\rk w_i > r$.
(This decomposition is unique; in particular, $a_k w_k$ is the shortest suffix of~$w$ with rank~$r$.)
By the definition of~$d_{r+1}$, for all $i \in \{0, \ldots, k\}$ there exists $u_i \in \Sigma^*$ with $M(w_i) = M(u_i)$ and $|u_i| \le d_{r+1}$.
Then $M(w) = M(u_0 a_1 u_1 a_2 u_2 \cdots a_k u_k)$.

Define a new alphabet $\Sigma_r$ and a monoid morphism $M_r :
\Sigma_r^* \to \Q^{n \times n}$ with $M_r(\Sigma_r) = \{M(a_i u_i) : i
\in \{1, \ldots, k\}\}$, and note that $\rk M_r(b)=r$ for all $b\in \Sigma_r$.
Then there is a word $y \in \Sigma_r^*$ such that $M_r(y) = M(a_1 u_1 \cdots a_k u_k)$.
By \autoref{prop-max-rank}, there is $x \in \Sigma_r^*$ with $M_r(y) = M_r(x)$ and $|x| \le h$.
Obtain the word $v \in \Sigma^*$ from~$x$ by replacing each letter $b \in \Sigma_r$ in~$x$ by~$a_i u_i$ for $i\in \{1, \ldots, k\}$ such that $M_r(b) = M(a_i u_i)$.
Then $M_r(x) = M(v)$, and thus $M(w) = M(u_0 a_1 u_1 \cdots a_k u_k) = M(u_0) M_r(y) = M(u_0) M_r(x) = M(u_0) M(v) = M(u_0 v)$, where $|u_0 v| = |u_0| + |v| \le d_{r+1} + (d_{r+1} + 1) |x| \le d_{r+1} + (d_{r+1} + 1) h$.
\end{proof}

We can now prove our main result.

\begin{proof}[Proof of \autoref{thm-main}]
We prove by induction that for all $r \in \{0, \ldots, n\}$ we have $d_r \le (h+1)^{n-r} d_n + (h+1)^{n-r} - 1$.
For the base case, $r=n$, this is trivial.
For the step, let $r < n$.
We have:
\begin{align*}
d_r \
&\le \ h + (h+1) d_{r+1} && \text{(\autoref{prop-induction-step})} \\
&\le \ h + (h+1) \left( (h+1)^{n-r-1} d_n + (h+1)^{n-r-1} - 1 \right) && \text{(induction hypothesis)} \\
&=   \ h + (h+1)^{n-r} d_n + (h+1)^{n-r} - h - 1
\end{align*}
This completes the induction proof.
Hence $d_0 \le (h+1)^n (d_n + 1) = 2^{n(2 n + 3)} g(n)^n (d_n + 1)$.
The rank-$n$ matrices in $M(\Sigma)$ generate a finite subgroup of~$\GL(n,\Q)$.
So it follows by \autoref{rem-group-diameter} that $d_n + 1 \le g(n)$.
Thus $d_0 \le \ourbound$.
\end{proof}

\section{Algorithmic Applications}\label{sec-applications}

\autoref{thm-main-intro} gives an exponential-space algorithm for
deciding finiteness of a finitely generated rational matrix
semigroup. In fact, the following theorem shows that 
deciding finiteness is in the second level of the weak EXP hierarchy
(see e.g.~\cite{Hemachandra89} for a definition).
\begin{ourtheorem}\label{thm-complexity}
  Given be a finite set $\M \subseteq \Q^{n\times n}$ of rational
  matrices, the problem of deciding finiteness of the generated
  semigroup $\overline{\M}$ is in $\textup{coNEXP}^{\textup{NP}}$.
\end{ourtheorem}
\begin{proof}
  For a $\textup{NEXP}^{\textup{NP}}$ algorithm deciding
  infiniteness, non-deterministically guess in exponential time some
  $M=M_1\cdots M_\ell$, $M_i\in \M$, with
  $\ell=2^{n(2n+3)}g(n)^{n+1}+1$ as a witness for infiniteness. Then,
  using a call to an NP oracle, check whether there are
  $M_1',\ldots,M_r' \in \M$ such that $M=M_1'\cdots M_r'$ for some
  $0\le r<\ell$. If the call is successful then reject, otherwise
  accept.

  Correctness of the algorithm immediately follows from
  \autoref{thm-main-intro}: if $\overline{\M}$ is finite, then the
  $M_1',\ldots,M_r'\in \M$ such that $M=M_1'\cdots M_r'$
  are guaranteed to exist.
\end{proof}
This is the first improvement of the non-elementary algorithm of
Mandel and Simon~\cite{MandelSimon77}.

Another immediate consequence of \autoref{thm-main-intro} is an
upper bound on the complexity of the membership problem for
finite matrix semigroups:
\begin{ourtheorem}
  Given a finite set of rational matrices $\M\subseteq \Q^{n\times n}$
  such that $\overline{\M}$ is finite and $A \in \Q^{n\times n}$, the
  problem of deciding whether $A\in \overline{\M}$ is in
  \textup{NEXP}.
\end{ourtheorem}

In the remainder of this section, we discuss implications of
\autoref{thm-complexity} to decision problems in automata theory.

\subsection{Weighted Automata}
The motivation for Mandel and Simon to study the finiteness problem originated from investigating the decidability of the finiteness problem (originally called boundedness problem in~\cite{MandelSimon77}) for weighted automata.
A \emph{weighted automaton over~$\Q$} is a quintuple $\A = (n, \Sigma, M, \alpha, \eta)$ where $n \in \N$ is the number of states, $\Sigma$ is the finite alphabet, $M : \Sigma \to \Q^{n \times n}$ maps letters to transition matrices, $\alpha \in \Q^n$ is the initial state vector, and $\eta \in \Q^n$ is the final state vector.
We extend~$M$ to the monoid morphism $M : \Sigma^* \to \Q^{n \times n}$ as before.
Such an automaton defines a function $|\A| : \Sigma^* \to \Q$ by defining $|\A|(w) = \alpha M(w) \eta^T$, where the superscript~$T$ denotes transpose.
We say, $\A$ is \emph{finite} if the image of~$|\A|$ is finite, i.e., if $|\A|(\Sigma^*) \subseteq \Q$ is a finite set.
The \emph{finiteness problem} asks whether a given automaton is finite.

It is clear that if $M(\Sigma^*)$ is finite then $\A$ is finite.
The converse is not generally true: e.g., any automaton~$\A$ whose initial state vector is the zero vector satisfies $|\A|(\Sigma^*) = \{0\}$, hence is finite, regardless of $M(\Sigma^*)$.
However, it is argued in the proof of Corollary~5.4 in~\cite{MandelSimon77} that, given an automaton~$\A$, one can compute, in exponential time, a polynomial-size automaton~$\B$ with monoid morphism~$M_\B$ such that (i) $|\A| = |\B|$, and (ii) $\A$ (and hence $\B$) is finite if and only if $M_\B(\Sigma^*)$ is finite.%
\footnote{We remark that this automaton~$\B$ has the minimal number of states among the automata defining the function~$|\A|$.
This minimal automaton goes back to \cite{Schutzenberger61} and has been further studied in, e.g., \cite{CarlylePaz71,Fliess74}.}
Mandel and Simon use this argument to show that the finiteness problem for weighted automata over~$\Q$ is decidable. \autoref{thm-complexity} then immediately gives:
\begin{ourcorollary}
The finiteness problem for weighted automata over~$\Q$ can be decided in $\textup{coNEXP}^{\textup{NP}}$.
\end{ourcorollary}

\subsection{Affine Integer Vector Addition Systems with States}

We show that \autoref{thm-main-intro} together with
\autoref{cor-card-bound} imply an upper bound for the reachability
problem in affine integer vector addition systems with states with the
finite monoid property (\afmp) studied in~\cite{BHM18}. An affine
$\Z$-VASS in dimension $d\in \N$ is a tuple $\V=(d,Q,T)$ such that $Q$
is a finite set of states and $T\subseteq Q\times \Z^{d\times d}
\times \Z^d \times Q$ is a finite transition relation. Setting $\M :=
\{ A\in \Z^{d\times d} : (q,A,\vec b,r)\in T \}$, in \afmp\ we
additionally require that $\overline{\M}$ is finite. A configuration
of $\V$ is a tuple $(q,\vec v) \in Q \times \Z^d$ which we write as
$q(\vec v)$.  We define the step relation ${\rightarrow} \subseteq (Q
\times \Z^d)^2$ such that $q(\vec v) \rightarrow r(\vec w)$ if and
only if there is a transition $(q,A,\vec b,r)\in T$ such that $\vec w
= A\cdot \vec v + \vec b$. Moreover, we denote by $\rightarrow^*$ the
reflexive transitive closure of $\rightarrow$.  For a configuration
$q(\vec v)$, we define the reachability set of $q(\vec v)$ as
$\reach(q(\vec v)) := \{ r(\vec w) : q(\vec v) \rightarrow^* r(\vec w)
\}$. Given configurations $q(\vec v)$ and $r(\vec w)$, reachability is
the problem of deciding whether $r(\vec w) \in \reach(q(\vec
v))$. Note that $\reach(q(\vec v))$ is in general infinite despite
$\overline{\M}$ being finite.

The reachability problem for \afmp\ was shown decidable
in~\cite{BHM18} by a reduction to reachability in $\Z$-VASS. A
$\Z$-VASS is an \afmp\ in which every transition is of the form
$(q,I_d,\vec b,r)$. The reachability problem for $\Z$-VASS is known to
be NP-complete, see e.g.~\cite{HH14}. The size of the $\Z$-VASS
obtained in the reduction given in~\cite{BHM18} grows in $|\mon{\M}|$
and hence leads to a non-elementary upper bound for reachability in
\afmp\ assuming Mandel and Simon's bound. The results of this paper
enable us to significantly improve this upper bound.
\begin{ourcorollary}
  The reachability problem for \afmp\ can be decided in
  $\textup{EXPSPACE}$.
\end{ourcorollary}
\begin{proof}
  Let $\V=(d,Q,T)$ be an \afmp\ and let $\M$ be defined as above. Set
  $||\M|| := |\mon \M| \cdot d^2 \cdot \max \{ \log(||A|| +1) : A \in
  \mon\M \}$, where $||A||$ is the largest absolute value of all
  entries of $A$. Since $||A_1\cdots A_n|| \le
  d^n\cdot||A_1||\cdots||A_n||$ for all $A_1,\ldots,A_n \in
  \Z^{d\times d}$ and $n\in \N$, by \autoref{thm-main-intro} and
  \autoref{cor-card-bound} we have
  \[
  ||\M|| \le |T|^{2^{O(d^2\cdot \log d)}} \cdot d^2 \cdot
  {2^{d(2d+3)} g(d)^{d+1}} \cdot (\log d + ||T||) \le
  ||T||^{2^{O(d^2\cdot \log d)}},
  \]
  where $||T|| := \sum_{(q,A,\vec b,r)\in T} d^2 \cdot \lceil \log
  ||A|| + \log ||\vec b|| + 1\rceil$. It can be deduced from the proof
  of~\cite[Thm.~7]{BHM18} that reachability in $\V$ can be decided in
  non-deterministic space that is polynomially bounded in the encoding
  of $\V$ and poly-logarithmically in $||\M||$, from which the desired
  exponential space upper bound follows.
\end{proof}

\section{Conclusion}
The main result of this paper has been to show that any element in the
finite multiplicative semigroup $\mon \M$ generated by a finite set
$\M$ of $m$ rational $n\times n$ matrices can be obtained as a product
of generators of length at most $2^{O(n^2 \log n)}$. This length bound
immediately gives that $|\mon \M|$ is bounded by $m^{2^{O(n^2 \log
    n)}}$.

There remain two immediate questions that we did not answer in this
article. The first is whether the order of growth of $|\mon \M|$ we
obtained is tight. If $\mon \M$ is a group its order can be bounded by
$2^nn!$ for almost all $n$, and this bound is attained by the group of
signed permutation matrices. In contrast, in the semigroup case $|\mon
\M|$ also depends on $m$. We conjecture that our doubly exponential
upper bound is not optimal and that it is possible to establish an
exponential upper bound of $|\mon \M|$ in terms of $m$ and $n$. The
second open question concerns the precise complexity of deciding
finiteness of matrix semigroups. We have been unable to establish any
non-trivial lower bounds on this problem and conjecture that our
$\textup{coNEXP}^{\textup{NP}}$ upper bound can significantly be
improved, possibly by adapting techniques of Babai et
al.~\cite{BabaiBealsRockmore93}.

\bibliographystyle{plain}
\bibliography{lit}

\newpage

\appendix

\section{Rational vs.\ Integer Matrix Groups}

We show that any finite subgroup of~$\GL(n,\Q)$ is conjugate to a finite subgroup of~$\GL(n,\Z)$.
This implies that rational matrix groups cannot be larger than integer matrix groups.

First we need a basic fact about finitely generated Abelian groups.
Let $A$ be an (additively written) Abelian (= commutative) group.
The group~$A$ is isomorphic to~$\Z^n$ (``\emph{free of rank~$n$}'') if and only if there is a set $\{a_1, \ldots, a_n\} \subseteq A$, called \emph{basis} of~$A$, such that for each $g \in A$ there are unique coefficients $k_1, \ldots, k_n \in \Z$ such that $g = k_1 a_1 + \cdots + k_n a_n$.
For any $g_1, \ldots, g_m \in A$ we write
\[
 \spa{g_1, \ldots, g_m} \ := \ \left\{ k_1 g_1 + \cdots + k_m g_m : k_1, \ldots, k_m \in \Z \right\}\,.
\]
Using elementary arguments, we prove the following proposition, which is related to the Fundamental Theorem of Finitely Generated Abelian Groups.
\begin{lemma} \label{lem-subgroups-of-Zn}
Let $A$ be isomorphic to~$\Z^n$, and let $H$ be a subgroup of~$A$.
Then $H$ is isomorphic to~$\Z^m$ for some $m \le n$.
\end{lemma}
\begin{proof}
We proceed by induction on~$n$, the rank of~$A$.
For the base case, if $n=0$, then $H = A = \{0\}$.
For the induction step, suppose the theorem holds for ranks less than~$n$.
Fix a basis $\{a_1,\dots,a_n\}$ of~$A$ and define
$\phi : H \rightarrow \Z$ by $\phi(k_1 a_1 + \dots + k_n a_n) := k_1$.

If $\phi(H) = \{0\}$ then $H \subseteq \spa{a_2, \dots, a_n}$, so $H$~is isomorphic to $\Z^m$ for some $m \le n-1$ by the induction hypothesis.
So assume $\phi(H) \ne \{0\}$.
Let $m \in \N\setminus\{0\}$ be the smallest positive value in~$\phi(H)$ and let $h_1 \in H$ be such that $\phi(h_1) = m$.
Note that $h_1 \notin \ker(\phi)$.

Consider any $h \in H$.
Let $\ell, j, k_2, \ldots, k_n \in \Z$ with $j \in \{0, \ldots, m-1\}$ such that $h = (\ell m + j) a_1 + k_2 a_2 + \dots + k_n a_n$.
Then $\phi(h - \ell h_1) = j < m$.
By the definition of~$m$ it follows that $j = 0$, and thus $h = \ell h_1 + h_2$ with $h_2 \in \ker(\phi)$.

Let $H_1 := \spa{h_1}$ and $H_2 := \ker(\phi)$.
They are subgroups of~$H$.
By the argument in the previous paragraph, we have $H = H_1 + H_2$.
Further $H_1 \cap H_2 = \{0\}$, as $h_1 \notin \ker(\phi)$.
Since $H_2$ is a subgroup of $\spa{a_2, \ldots, a_n}$, by the induction hypothesis, $H_2$ is isomorphic to $\Z^{m'}$ for some $m' \le n-1$.
It follows that $H$ is isomorphic to $H_1 \times \Z^{m'}$, and so to $\Z^{1+m'}$.
\end{proof}

For any group $G \subseteq \GL(n,\Q)$ and any $C \in \GL(n,\Q)$, the set $C G C^{-1}$ is a group that is conjugate, hence isomorphic, to~$G$.
Following~\cite{KuzmanovichPavlichenkov02} we show that any finite rational matrix group is conjugate to an integer matrix group.
\begin{proposition}
Let $G \subseteq \GL(n,\Q)$ be finite.
Then there is $C \in \GL(n,\Q)$ such that $C G C^{-1} \subseteq \GL(n,\Z)$.
\end{proposition}
\begin{proof}
Observe that for any $M \in G$, we have $G M = G$.
Define $A := \sum_{M \in G} \Z^n M \subseteq \Q^n$. 
By the observation, $A M = A$ holds for all $M \in G$.

The set~$A$ forms a group with respect to vector addition.
It is finitely generated by the rows of the matrices in~$G$, i.e., $A = \spa{e_i M : M \in G,\ i \in \{1, \ldots, n\}}$, where $\{e_1, \ldots, e_n\} \subseteq \{0,1\}^n$ is the standard basis.
Let $d \in \N$ be a common denominator of all entries of all those generators.
Then $d A \subseteq \Z^n$, so by \autoref{lem-subgroups-of-Zn} the group $A$ is isomorphic to $\Z^{n'}$ for some $n' \le n$.
On the other hand, since the identity matrix is in~$G$, we have $\Z^n \subseteq A$.
Hence, $n \le n'$, and so $A$~is isomorphic to~$\Z^n$.
Let $\gamma : \Z^n \to A$ be such an isomorphism.
Then there is a matrix $C \in \GL(n,\Q)$ such that $v C = \gamma(v)$ holds for all $v \in \Z^n$.

Let $M \in G$.
Since $\Z^n C M C^{-1} = A M C^{-1} = A C^{-1} \subseteq \Z^n$, all entries of $C M C^{-1}$ are integers.
\end{proof}

\end{document}